\documentclass[12pt]{article}
\usepackage{amssymb}
\usepackage{amsmath,amsthm}
\usepackage[latin1]{inputenc}
\usepackage[dvips]{graphicx}
\usepackage{hyperref}
\usepackage{color}
\usepackage{enumerate}

\hypersetup{colorlinks=true}

\hypersetup{colorlinks=true, linkcolor=blue, citecolor=blue,urlcolor=blue}

 \newtheorem{remark}{Remark}

 \newtheorem{lemma}[remark]{Lemma}
 \newtheorem{theorem}[remark]{Theorem}
 
 \newtheorem{corollary}[remark]{Corollary}

\title{On the strong metric dimension of corona product graphs and join graphs}

\author{  Dorota Kuziak$^{(1)}$, Ismael G. Yero$^{(2)}$ and Juan A.
Rodr\'{\i}guez-Vel\'{a}zquez$^{(1)}$
    \\
$^{(1)}${\small Departament d'Enginyeria Inform\`atica i Matem\`atiques,}\\
{\small Universitat Rovira i Virgili,}  {\small Av. Pa\"{\i}sos
Catalans 26, 43007 Tarragona, Spain.} \\{\small
juanalberto.rodriguez\@@urv.cat, dorota.kuziak\@@urv.cat}
\\
$^{(2)}${\small Departamento de Matem\'aticas, Escuela Polit\'ecnica Superior de Algeciras}\\
{\small Universidad de C\'adiz,} {\small
Av. Ram\'on Puyol s/n, 11202 Algeciras, Spain.} \\ {\small
ismael.gonzalez\@@uca.es}\\
}

\begin{document}
\maketitle

\begin{abstract}
Let $G$ be a connected graph. A vertex $w$  strongly resolves a pair $u$, $v$ of vertices of $G$ if there exists some shortest $u-w$ path containing $v$ or some shortest $v-w$ path containing $u$. A set $W$ of vertices is a strong resolving set for $G$ if every pair of
vertices of $G$ is strongly resolved by some vertex of $W$. The smallest cardinality of a strong resolving set for $G$ is called the strong metric
dimension of $G$. It is known that the problem of computing this invariant is NP-hard. It is therefore desirable  to reduce the problem of computing the strong metric dimension of product graphs, to the problem of computing some parameter of the factor graphs. We show that the problem of finding the strong metric dimension of the corona product $G\odot H$, of two graphs $G$ and $H$, can be transformed to the problem of finding certain clique number of $H$. As a consequence of the study we show that if $H$ has diameter two, then the strong metric dimension of $G\odot H$ is obtained from the strong metric dimension of $H$ and, if $H$ is not connected or its  diameter is greater than two, then the strong metric dimension of $G\odot H$ is obtained from the strong metric dimension of $K_1\odot H$, where $K_1$ denotes the trivial graph. The strong metric dimension of join graphs is also studied.
\end{abstract}



\section{Introduction}

Generators of metric spaces are sets of points with the property that every point of the space is uniquely determined by the distances from their elements. Such generators put a light on some kinds of problems in graph theory  that apparently are not directly related to metric spaces. Given a simple and connected graph $G=(V,E)$, we consider the metric $d_G:V\times V\rightarrow \mathbb{R}^+$, where $d_G(x,y)$ is the length of a shortest path between $u$ and $v$. $(V,d_G)$ is clearly a metric space. A vertex $v\in V$ is said to distinguish two vertices $x$ and $y$ if $d_G(v,x)\ne d_G(v,y)$.
A set $S\subset V$ is said to be a \emph{metric generator} for $G$ if any pair of vertices of $G$ is
distinguished by some element of $S$. A minimum generator is called a \emph{metric basis}, and
its cardinality the \emph{metric dimension} of $G$, denoted by $dim(G)$. Motivated by the problem of uniquely determining the location of an intruder in a network, the concept of metric
dimension of a graph was introduced by Slater in \cite{slater-ld,slater}, where the metric generators were called \emph{locating sets}. The concept of metric dimension of a graph was also introduced by Harary and Melter in \cite{harary}, where metric generators were called \emph{resolving sets}. Applications
of this invariant to the navigation of robots in networks are discussed in \cite{Khuller} and applications to chemistry in \cite{pharmacy1,pharmacy2}.  Once the first article in this topic was published several papers have been appearing in the literature, e.g. \cite{bailey, pelayo, chartrand, feng, guo, haynes, melter, zhang, ykr}.  Remarkable variations of the concept of metric generators are  resolving dominating sets \cite{brigham}, independent resolving sets \cite{chartrand1}, local metric sets \cite{LocalMetric}, strong resolving sets \cite{seb}, etc.

In this article we are interested in the study  of strong resolving sets \cite{oellerman,seb}.
 For two vertices $u$ and $v$ in a connected graph $G$, the interval $I_G [u, v]$ between $u$ and $v$ is defined as the collection of all vertices that belong to some
shortest $u-v$ path. A vertex $w$ strongly resolves two vertices $u$ and $v$ if $v\in  I_G [u,w]$ or  $u\in I_G [v,w]$. A set $S$ of vertices
in a connected graph $G$ is a \emph{strong resolving set} for $G$ if every two vertices of $G$ are strongly resolved by some vertex
of $S$. The smallest cardinality of a strong resolving set of $G$ is called \emph{strong metric dimension} and is denoted by
$dim_s(G)$.  So, for example, $dim_s(G)=n-1$ if and only if $G$ is the complete graph  of order $n$.
For the cycle $C_n$ of order $n$ the strong dimension is $dim_s(C_n) = \lceil n/2\rceil$  and if $T$ is a tree, its strong metric dimension equals the
number of leaves of $T$ minus $1$ (see \cite{seb}).  We say that a strong resolving set for $G$ of cardinality $dim_s(G)$ is a \emph{strong metric basis} of $G$.

It is known that the problem of computing the strong metric dimension of a graph is NP-hard \cite{oellerman}. It is therefore desirable  to reduce the problem of computing the strong metric dimension of product graphs, to the problem of computing some parameters of the factor graphs. We show that the problem of finding the strong metric dimension of the corona product $G\odot H$, of two graphs $G$ and $H$, can be transformed to the problem of finding certain clique number of $H$. As a consequence of the study we show that if $H$ has diameter two, then the strong metric dimension of $G\odot H$ is obtained from the strong metric dimension of $H$ and, if $H$ is not connected or its  diameter is greater than two, then the strong metric dimension of $G\odot H$ is obtained from the strong metric dimension of $K_1\odot H$, where $K_1$ denotes the trivial graph. The strong metric dimension of join graphs is also studied.

We begin by giving some basic concepts and notations. For two adjacent vertices $u$ and $v$ of $G=(V,E)$ we use the notation  $u\sim v$. For a  vertex
$v$ of $G$, $N_G(v)$ denotes the set of neighbors that $v$ has in $G$, i.e.,  $N_G(v)=\{u\in V:\; u\sim v\}$. The set $N_G(v)$ is called the \emph{open neighborhood of} $v$ in $G$  and $N_G[v]=N_G(v)\cup \{v\}$ is called the \emph{closed neighborhood of} $v$ in $G$.  The degree of a vertex $v$ of $G$ will be denoted by $\delta_G(v)$, i.e., $\delta_G(v)=|N_G(v) |$. Recall that the \emph{clique number} of a graph $G$, denoted by $\omega(G)$, is the number of vertices in a maximum clique in $G$.
Two distinct vertices $x$, $y$ are called \emph{true twins} if $N_G[x] = N_G[y]$. We say that $X\subset V$ is a \emph{twin-free clique} in $G$ if the subgraph induced by $X$ is a clique and for every $u,v\in X$ it follows $N_G[u]\ne N_G[v]$, i.e., the subgraph induced by $X$ is a clique and it contains no true twins. We say that the \emph{twin-free clique number} of $G$, denoted by $\varpi(G)$, is the maximum cardinality among all twin-free cliques in $G$. So, $\omega(G)\ge \varpi(G)$. We refer to a  $\varpi(G)$-set in a graph $G$ as a twin-free clique  of cardinality $\varpi(G)$.

We say that a vertex $u$ of $G$ is \emph{maximally distant} from $v$ if for every $w\in N_G(u)$, $d_G(v,w)\le d_G(u, v)$. If $u$ is
maximally distant from $v$ and $v$ is maximally distant from $u$, then we say that $u$ and $v$ are \emph{mutually maximally distant}. Since no vertex of $G$ strongly resolves two mutually maximally distant vertices of $G$, we have the following remark which will be useful later.
\begin{remark}\label{lemmaL}
For every pair of mutually maximally distant vertices $x,y$ of a connected graph $G$ and for every strong metric basis $S$ of $G$, it follows that  $x\in S$ or $y\in S$.
\end{remark}
Let $G$ and $H$ be two graphs of order $n_1$ and $n_2$, respectively. Recall that the corona product $G\odot H$ is defined as the graph obtained from $G$ and $H$ by taking one copy of $G$ and $n_1$ copies of $H$ and joining by an edge each vertex from the $i^{th}$-copy of $H$ with the $i^{th}$-vertex of $G$. We will denote by $V=\{v_1,v_2,...,v_n\}$ the set of vertices of $G$ and by $H_i=(V_i,E_i)$ the copy of $H$ such that $v_i\sim v$ for every $v\in V_i$. The join $G+H$ is defined as the graph obtained from disjoined graphs $G$ and $H$ by taking one copy of $G$ and one copy of $H$ and joining by an edge each vertex of $G$ with each vertex of $H$. Notice that the corona graph $K_1\odot H$ is isomorphic to the join graph $K_1+H$.

\section{Main results}

We shall start studying the relationship between the strong metric dimension of a connected graph and its twin-free clique number.

\begin{theorem}\label{TheoremTwinclique}
Let $H$ be a  connected graph of order $n\ge 2$. Then
$$dim_s(H)\le n-\varpi(H).$$
Moreover, if $H$ has diameter two, then
$$dim_s(H) = n-\varpi(H).$$
\end{theorem}

\begin{proof}
Let $W$ be a maximum twin-free clique in $H=(V,E)$. We will show that $V - W$ is a strong resolving set for $H$. Since $W$ is a twin-free clique, for any two distinct vertices $u,v\in W$ there exists $s\in V - W$ such that either ($s\in N_H(u)$ and $s\notin N_H(v)$) or ($s\in N_H(v)$ and $s\notin N_H(u)$). Without loss of generality, we consider $s\in N_H(u)$ and $s\notin N_H(v)$. Thus, $u\in I_H[v,s]$ and, as a consequence,  $s$ strongly resolves $u$ and $v$. Therefore, $dim_s(H)\le n - \varpi (H)$.

Now, suppose that $H$ has diameter two. Let $X$ be a strong metric basis of $H$ and let $u$, $v$ be two distinct vertices of $H$. If $d_H(u,v)=2$ or $N_H[u]=N_H[v]$, then $u$ and $v$ are mutually maximally distant vertices of $H$, so $u\in X$ or $v\in X$. Hence, for any two distinct vertices $x,y\in V-X$ we have $x\sim y$ and $N_H(x)\ne N_H(y)$. As a consequence, $|V-X|\le \varpi(H)$. Therefore, $dim_s(H)\ge n-\varpi(H)$ and the result follows.
\end{proof}

\begin{corollary}
Let $H$ be a graph of diameter two and order $n$. Let $\omega(H)$ be the clique number of $H$ and let $c(H)$ be the number of  vertices of $H$ having degree $n-1$.
If the only true twins of $H$ are  vertices of degree $n-1$, then
  $$dim_s(H)=n+c(H)-\omega(H)-1.$$
Moreover,  if $H$ has no true twins, then $$dim_s(H)=n-\omega(H).$$
\end{corollary}

\begin{lemma}\label{lemma_join_clique}
Let $G$ and $H$ be two connected graphs of order $n_1\ge 2$ and $n_2\ge 2$, and maximum degree $\Delta_1$ and $\Delta_2$, respectively.
\begin{itemize}
\item[{\rm (i)}]  If $\Delta_1\ne n_1-1$ or $\Delta_2\ne n_2-1$, then
$$\varpi (G+H) = \varpi (G) + \varpi (H).$$
\item[{\rm (ii)}] If  $\Delta_1 = n_1-1$ and $\Delta_2 = n_2-1$, then
$$\varpi (G+H) = \varpi (G) + \varpi (H) - 1.$$
\end{itemize}
\end{lemma}

\begin{proof}
Given a $\varpi(G+H)$-set $Z$ we have that for every $u_1,u_2\in U=Z\cap V(G)$ it follows $N_{G+H}[u_1]\ne N_{G+H}[u_2]$. So, $N_{G}[u_1]\ne N_{G}[u_2]$ and, as a consequence, $U$ is a twin-free clique in $G$. Analogously we show that $W=Z\cap V(H)$ is a twin-free clique in $H$. Hence, $\varpi(G+H)=|Z|=|U|+|W|\le \varpi(G)+\varpi(H)$.

Now, if $\Delta_1 = n_1-1$ and $\Delta_2 = n_2-1$, then every $\varpi(G)$-set ($\varpi(H)$-set) contains exactly one vertex of degree $\Delta_1 = n_1-1$ ($\Delta_2 = n_2-1$) and every $\varpi(G+H)$-set contains exactly one vertex of degree $n_1+n_2-1$. Hence, in this case $|U|<\varpi(G)$ or $|W|<\varpi(H)$ and, as a consequence, $\varpi(G+H)=|Z|=|U|+|W|\le \varpi(G)+\varpi(H)-1$.

On the other hand, let $U'$ be a $\varpi(G)$-set and let $W'$ be a $\varpi(H)$-set.

In order to complete the proof of (i), we assume, without loss of generality, that $\Delta_1\ne n_1-1$. Let $u\in  U'$ and $w\in  W'$. Since $\delta_G(u)\ne n_1 - 1$, there exists a vertex $x\in V(G)-U'$ such that $u\not\sim x$. From the definition of $G+H$ we have $w\sim x$ and  the subgraph induced by $U'\cup W'$ is a clique in $G+H$. So, $u$ and $w$ are not true twins in $G+H$ and, as a consequence, $U'\cup W'$ is a twin-free clique in $G+H$. Hence, $\varpi(G+H)\ge |U'\cup W'|= \varpi(G)+\varpi(H)$. The proof of (i) is complete.

Now, if $\Delta_1 = n_1-1$,
then we take $x\in  U'$ such that $\delta_{G}(x) = n_1 - 1$ 
and as above we see that two vertices  $v,w\in U'\cup W'-\{x\}$    are not true twins in $G+H$. Hence, $U'\cup W'-\{x\}$ is a twin-free clique in $G+H$. So, $\varpi(G+H)\ge |U'|+|W'|-1= \varpi(G)+\varpi(H)-1$. Therefore, the proof of (ii) is complete.
\end{proof}

If $G$ and $H$ are two complete graphs of order $n_1$ and $n_2$, respectively, then $G+H = K_{n_1+n_2}$ and $dim_s(G+H) = dim_s(K_{n_1+n_2}) = n_1 + n_2 - 1$. From Theorem \ref{TheoremTwinclique} and Lemma \ref{lemma_join_clique} we obtain the following results.

\begin{theorem}\label{corollary_inequality}
Let $G$ and $H$ be two connected graphs of order $n_1\ge 2$ and $n_2\ge 2$, and maximum degree $\Delta_1$ and $\Delta_2$, respectively.

\begin{enumerate}[{\rm (i)}]
 \item If $\Delta_1\ne n_1-1$ or $\Delta_2\ne n_2-1$, then $$dim_s(G+H)=n_1 + n_2 - \varpi(G) - \varpi(H)\ge  dim_s(G)+dim_s(H).$$
 \item If $G$ and $H$ are graphs of diameter two where $\Delta_1\ne n_1-1$ or $\Delta_2\ne n_2-1$, then $$dim_s(G+H) = dim_s(G) + dim_s(H).$$
\item  If  $\Delta_1 = n_1-1$ and $\Delta_2 = n_2-1$,  then $$dim_s(G+H) = dim_s(G) + dim_s(H)+1.$$
\end{enumerate}
\end{theorem}







The following lemma shows that the problem of finding the strong metric dimension of a corona product graph can be transformed to the problem of finding the strong metric dimension of a graph of diameter two.

\begin{lemma}\label{iguales}
Let $G$ be a connected graph of order $n$ and let $H$ be a graph. Let $H_i$ be the subgraph of $G\odot H$ corresponding to the $i^{th}$-copy of $H$. Then
$$dim_s(G\odot H)=dim_s(K_1 + \bigcup_{i=1}^nH_i).$$
\end{lemma}

\begin{proof}
As the result is obvious for $n=1$, we take $n\ge 2$.
Let $v$ be the vertex of $K_1$ and let $S'$ be a strong resolving set for $G\odot H$.
 We will show that $S=\cup_{i=1}^n (S'\cap V_i)$ is a strong resolving set for $K_1 + \cup_{i=1}^nH_i$. We consider $x,y$ are two different vertices of $K_1 + \cup_{i=1}^nH_i$ not belonging to $S$.  We differentiate the following cases.

Case 1: $x=v$ and $y\in V_i$, for some $i$. For any $u\in V_j$, $j\ne i$, we have $x\in I_{K_1 + \cup_{i=1}^nH_i}[u,y]$ and since $y$ and $u$ are mutually maximally distant in $G\odot H$,  we have $y\in S$ or $u\in S$.

Case 2: $x,y\in V_i$. Let $u$ be a vertex of $S'$ which strongly resolves $x$ and $y$ in $G\odot H$. As no vertex of $G\odot H$ not belonging to $V_i$ strongly resolves $x$ and $y$, we have that $u\in V_i$ and $u\in S$. Hence,  $u$ strongly resolves $x$ and $y$ in $K_1 + \cup_{i=1}^nH_i$.

Note that in the case $x\in V_i$ and $y\in V_j$, $i\ne j$, we have that $x$ and $y$ are mutually maximally distant in $G\odot H$. Thus,  we have $x\in S$ or $y\in S$. Hence, $S$ is a strong resolving set for  $K_1 + \cup_{i=1}^nH_i$ and, as a consequence, $dim_s(G\odot H)\ge dim_s(K_1 + \cup_{i=1}^nH_i).$

Now, given a strong resolving set for $K_1 + \cup_{i=1}^nH_i$ denoted by  $W'$,  let us show that $W=W'-\{v\}$ is a strong resolving set for $G\odot H$.
Let $x,y$ be two different vertices of $G\odot H$ not belonging to $W$. We denote by  $V=\{v_1,v_1,...,v_n\}$  the vertex set of $G$, where $v_i$ is the vertex of $G$ adjacent to every vertex of $V_i$ in $G\odot H$, $i\in \{1,...,n\}$. We differentiate the following cases.

Case 1: $x=v_i\in V$ and $y\in V_i$. Let $u\in V_j$, $j\ne i$. In this case we have $x\in I_{G\odot H}[u,y]$ and, since $y$ and $u$ are mutually maximally distant in $K_1 + \cup_{i=1}^nH_i$, we have $y\in W$ or $u\in W$.

Case 2. $x=v_i\in V$ and $y\in V_j$, $j\ne i$. For every $u\in V_i$ we have  $x\in I_{G\odot H}[u,y]$ and, since $y$ and $u$ are mutually maximally distant in $K_1 + \cup_{i=1}^nH_i$,  we have $y\in W$ or $u\in W$.

Case 3: $x,y\in V$. Let $x=v_i$, $y=v_j$, $u_i\in V_i$ and $u_j\in V_j$. We have $x\in I_{G\odot H}[u_i,y]$ and $y\in I_{G\odot H}[u_j,x]$. As $u_i$ and $u_j$ are mutually maximally distant in $K_1 + \cup_{i=1}^nH_i$,  we have $u_i\in W$ or $u_j\in W$.

Finally, note that the case  $x\in V_i$ and $y\in V_j$, where $i,j\in\{1,2,...,n\}$, leads to $x\in W$ or $y\in W$. Therefore,  $W$ is a strong resolving set for  $G\odot H$ and, as a consequence, $dim_s(G\odot H)\le dim_s(K_1 + \cup_{i=1}^nH_i).$
\end{proof}

\begin{corollary}\label{theoremoK1}
For any connected graph $G$ of order $n$,
$dim_s(G\odot K_1)=n-1.$
\end{corollary}
\begin{proof}
For $H\cong K_1$ Lemma \ref{iguales} leads to  $dim_s(G\odot K_1)=dim_s(K_1 + \cup_{i=1}^nK_1)=dim_s(K_{1,n})=n-1.$
\end{proof}

Our next result is obtained from Lemma \ref{iguales} and Theorem \ref{TheoremTwinclique}.

\begin{theorem}\label{Equalities}
Let $G$ be a connected graph of order $n_1$. Let $H$ be a  graph  of order  $n_2$ and  maximum degree $\Delta$.
\begin{enumerate}[{\rm (i)}]
\item   If $\Delta=n_2-1$, then $dim_s(K_1 + H)=n_2+1-\varpi(H).$
\item  If $\Delta\le n_2-2$ or $n_1\ge 2$, then $dim_s(G\odot H)=n_1n_2-\varpi(H).$
\end{enumerate}
\end{theorem}
\begin{proof}
Since (i) is trivial, we will prove (ii). For $\Delta=n_2-1$ we have  $\varpi\left( K_1 + \cup_{i=1}^{n_1}H_i  \right) {\buildrel {n_1> 1}\over =} \varpi(K_1 + H)+1=\varpi(H)+1$, while for  $\Delta\le n_2-2$ we have  $\varpi\left( K_1 + \cup_{i=1}^{n_1}H_i  \right) = \varpi(K_1 + H)=\varpi(H)+1$. So, by Lemma \ref{iguales} and Theorem \ref{TheoremTwinclique} we conclude the proof.
\end{proof}

Let us derive some consequences of the above result.

\begin{corollary}\label{thClique}
Let $G$ be a connected graph of order $n_1$ and let  $H$ be a  graph of order $n_2$, clique number $\omega(H)$ and  maximum degree $\Delta$. Let  $c(H)$ be the number of  vertices of $H$ having degree $n_2-1$.
\begin{enumerate}[{\rm (i)}]
\item  If $H$ has no true twins and $\Delta=n_2-1$, then $$dim_s(K_1 + H) =  n_2 +1-\omega(H).$$
\item  If $H$ has no true twins and $\Delta\le n_2-2$, $$dim_s(K_1 + H) =  n_2 -\omega(H).$$
\item  If $H$ has no true twins and $n_1\ge 2$, then $$dim_s(G\odot H) =  n_1n_2-\omega(H).$$
\item  If the only true twins of $H$ are  vertices of degree $n_2-1$, then  $$dim_s(K_1 + H) =  n_2 +c(H)-\omega(H)$$
\item  If the only true twins of $H$ are  vertices of degree $n_2-1$ and $n_1\ge 2$, then $$dim_s(G\odot H) =  n_1n_2+c(H)-1-\omega(H).$$
\end{enumerate}
\end{corollary}

As our next result shows, when $H$ is a triangle free graph we obtain the exact value for the strong metric dimension of $G\odot H$.

\begin{corollary}\label{upper:n-2}
Let $G$ be a connected graph of order $n_1$ and let  $H$ be a  triangle free graph of order $n_2\ge 3$ and maximum degree $\Delta$. If $n_1\ge 2$ or $\Delta \le n_2-2$, then  $$dim_s(G\odot H)= n_1n_2-2.$$
\end{corollary}

Our next result is an interesting consequence of Theorem \ref{TheoremTwinclique} and Theorem \ref{Equalities}.

\begin{theorem}
Let $G$ be a connected graph of order $n_1$. Let $H$ be a  graph  of order  $n_2$ and  maximum degree $\Delta$.
\begin{enumerate}[{\rm (i)}]
\item  If $\Delta=n_2-1$, then
$$dim_s(K_1 + H)=dim_s(H)+1.$$
\item  If $H$ has diameter two and  either $\Delta\le n_2-2$ or $n_1\ge 2$, then $$dim_s(G\odot H)=(n_1-1)n_2+dim_s(H).$$

\item  If $H$ is not connected or its  diameter is greater than two, then
$$dim_s(G\odot H)=(n_1-1)n_2+dim_s(K_1 + H).$$
\end{enumerate}
\end{theorem}

Note that the above theorem allow us to derive results on the strong metric dimension of some join graphs.

\begin{corollary}
Let $H$ be a graph of order $n$ and maximum degree $\Delta$.
\begin{enumerate}[{\rm (i)}]
\item   If $\Delta = n-1$, then $$dim_s(K_r + H)=dim_s(H)+r.$$
\item  If  $\Delta\le n-2$ and  $H$ has diameter two, then
$$dim_s(K_r + H)=dim_s(H)+r-1.$$
\item  If $H$ is not connected or its  diameter is greater than two, then $$dim_s(K_r + H)=dim_s(K_1 + H)+r-1.$$
\end{enumerate}
\end{corollary}

\subsection{Bounds}

It is well known that the second smallest Laplacian eigenvalue of a graph is probably the most important information
contained in the spectrum. This eigenvalue, frequently called \emph{algebraic connectivity}, is related to several important graph
invariants and imposes reasonably good bounds on the values of several parameters of graphs which are very hard to
compute.

The following theorem shows the relationship between the algebraic
connectivity of a graph and the clique number.

\begin{theorem}\label{teo-mu-strong}
Let $G$ be a connected non-complete graph of order $n$, maximum degree $\Delta$ and algebraic connectivity $\mu$.
The clique number of $\omega(G)$ is bounded by
$$\omega(G)\le \frac{n(\Delta-\mu+1)}{n-\mu}.$$
\end{theorem}

\begin{proof}
The algebraic connectivity of $G$,  satisfies the following equality shown by Fiedler
\cite{fiedler},
\begin{equation}
  \mu=2n \min \left\{ \frac{\sum_{v_i\sim v_j}(w_i-w_j)^2  }
  {\sum_{v_i\in V}\sum_{v_j\in V}(w_i-w_j)^2} \right\},     \label{rfiedler}
\end{equation}
where not all the components of the vector $(w_1, w_2, ..., w_n)\in \mathbb{R}^n$ are equal.
Let $S$ be a clique of $G=(V,E)$ of cardinality $\omega(G)$.
The vector  $w\in \mathbb{R}^n$ associated to $S$ is defined
as,
\begin{equation}\label{defin-de-w}
w_i= \left\lbrace \begin{array}{ll} 1  & {\rm if }\quad  v_i\in S;
                            \\ 0 &  {\rm  otherwise,} \end{array}
                                                \right.
\end{equation}
Considering the 2-partition $\{ S, V- S \}$ of the vertex set $V$ we have $(w_{i}-w_{j})^{2}=1$ if $v_{i}$ and $v_{j}$ \ are
in different sets of the partition, and 0 if they are in the same set. Then,
    \begin{equation}
      \sum_{v_i\in V}\sum_{v_j\in V}(w_i-w_j)^2=2\left| S\right|(n-\left| S\right|). \label{denominador}
    \end{equation}
By (\ref{rfiedler}) and (\ref{denominador}) we have
\begin{equation}
    \mu \le   \frac{n \sum_{v_i\sim v_j}(w_i-w_j)^2  }{\left| S\right|(n-\left| S\right|)}.    \label{myfiedler}
\end{equation}
Moreover, since $\displaystyle\sum_{v_i\sim v_j}(w_i-w_j)^2$ is the number of edges of $G$ having one endpoint in $S$ and the other one in $V-S$, we have
$\displaystyle\sum_{v_i\sim v_j}(w_i-w_j)^2=\displaystyle\sum_{v\in S}|N_{V-S}(v)|,$ where $N_{V-S}(v)$ denotes the set of neighbors that $v$ has in $V-S$.
Thus, since $S$  is a clique of $G$, we have that for every $v\in S$, $|N_{V-S}(v)|= \delta_G(v)-(|S|-1)$. Hence,
\begin{equation}\label{cond-mu-3}
 \mu \le   \frac{n \sum_{v\in S} \left(\delta_G(v)-|S|+1\right)}{\left| S\right|(n-\left| S\right|)}\le \frac{n(\Delta-|S|+1)}{n-|S|}.
\end{equation}
The result follows directly by inequality  (\ref{cond-mu-3}).
\end{proof}

The above bound is tight, it is achieved, for instance, for the Cartesian product graph $G=K_{r}\square K_2$, where $\mu=2$, $n=2r$, $\Delta=r$ and $\omega(G)=r.$

Notice that the above result and the inequality $\omega(H)\ge \varpi(H)$ combined with Theorem \ref{TheoremTwinclique}, Theorem \ref{corollary_inequality} or Theorem \ref{Equalities}, lead to lower bounds on the strong metric dimension. For instance, by Theorem \ref{TheoremTwinclique} we derive the following tight bound on the strong metric dimension of graphs with diameter two.

\begin{theorem}
Let $H$ be a connected graph of diameter two, order $n\ge 2$, maximum degree $\Delta$ and algebraic connectivity $\mu$. Then
$$dim_s(H)\ge \left\lceil\frac{n(n-\Delta-1)}{n-\mu}\right\rceil.$$
\end{theorem}

\section*{Acknowledgements}
The authors would like to thank the anonymous reviewers for their valuable comments and suggestions to improve the quality of the paper.


\begin{thebibliography}{99}
\bibitem{bailey} R.F. Bailey, K. Meagher, On the metric dimension of Grassmann graphs, \emph{Discrete Mathematics} \& \emph{Theoretical Computer Science} \textbf{13} (2011) 97--104.

\bibitem{brigham} R. C. Brigham, G. Chartrand, R. D. Dutton, P. Zhang, Resolving domination in graphs, \emph{Mathematica Bohemica} {\bf 128} (1) (2003) 25--36.


\bibitem{pelayo} J. C\'{a}ceres, C. Hernando, M. Mora, I. M. Pelayo, M. L. Puertas, C. Seara, D. R. Wood, On the metric dimension of Cartesian product of graphs, \emph{SIAM Journal on Discrete Mathematics} {\bf 21} (2) (2007) 273--302.

\bibitem{chartrand} G. Chartrand, L. Eroh, M. A. Johnson, O. R. Oellermann, Resolvability in graphs and the metric dimension of a graph, \emph{Discrete Applied Mathematics} {\bf 105} (2000) 99--113.

\bibitem{chartrand1} G. Chartrand, V. Saenpholphat, P. Zhang, The independent resolving number of a graph, \emph{Mathematica Bohemica} {\bf 128} (2003) 379--393.

\bibitem{feng} M. Feng, K. Wang, On the metric dimension of bilinear forms graphs, \emph{Discrete Mathematics} \textbf{312} (2012) 1266--1268.

\bibitem{fiedler}  M. Fiedler,  A property of eigenvectors of nonnegative symmetric matrices and its application to graph theory, \emph{Czechoslovak Mathematical Journal} {\bf  25} (100) (1975) 619--633.

\bibitem{guo} J. Guo, K. Wang, F. Li, Metric dimension of some distance-regular graphs, \emph{Journal of Combinatorial Optimization} DOI: 10.1007/s10878-012-9459-x.

\bibitem{harary} F. Harary, R. A. Melter, On the metric dimension of a graph, \emph{Ars Combinatoria} {\bf 2} (1976) 191--195.

\bibitem{haynes} T. W. Haynes, M. Henning, J. Howard, Locating and total dominating sets in trees, \emph{Discrete Applied Mathematics} {\bf 154} (2006) 1293--1300.

\bibitem{pharmacy1} M. A. Johnson, Structure-activity maps for visualizing the graph variables arising in drug design, \emph{Journal of Biopharmaceutical Statistics} {\bf 3} (1993) 203--236.

\bibitem{pharmacy2} M. A. Johnson, Browsable structure-activity datasets, \emph{Advances in Molecular Similarity} (R. Carbó--Dorca and P. Mezey, eds.) JAI Press Connecticut (1998) 153--170.


\bibitem{Khuller} S. Khuller, B. Raghavachari, A. Rosenfeld, Localization in graphs, Technical Report CS-TR-3326, University of Maryland at College Park, 1994.

\bibitem{melter} R.A. Melter, I. Tomescu, Metric bases in digital geometry, \emph{Computer Vision, Graphics, and Image Processing} {\bf 25} (1984) 113--121.

\bibitem{oellerman} O. R. Oellermann, J. Peters-Fransen, The strong metric dimension of graphs and digraphs, \emph{Discrete Applied Mathematics} {\bf 155} (2007) 356--364.

\bibitem{LocalMetric} F. Okamoto,   B. Phinezyn,  P. Zhang, The local metric dimension of a graph, \emph{ Mathematica Bohemica}  \textbf{135} (2010) (3)  239--255.

\bibitem{zhang} V. Saenpholphat, P. Zhang, Conditional resolvability in graphs: a survey, \emph{International Journal of Mathematics and Mathematical Sciences} {\bf 38} (2004) 1997--2017.

\bibitem{seb} A. Seb\H{o}, E. Tannier, On metric generators of graphs, \emph{Mathematics of Operations Research} {\bf 29} (2) (2004) 383--393.

\bibitem{slater-ld} P. J. Slater, Dominating and reference sets in a graph, \emph{Journal of Mathematical and Physical Sciences} {\bf 22} (1988) 445--455.

\bibitem{slater} P. J. Slater, Leaves of trees, Proceeding of the 6th Southeastern Conference on Combinatorics, Graph Theory, and Computing, \emph{Congressus Numerantium} {\bf 14} (1975) 549--559.

\bibitem{ykr} I. G. Yero, D. Kuziak, J. A. Rodr\'{\i}guez-Vel\'{a}zquez, On the metric dimension of corona product graphs, \emph{Computers} \& \emph{Mathematics with Applications} {\bf 61} (2011) 2793--2798.


\end{thebibliography}
\end{document}